\documentclass[11pt, leqno]{amsart}
\usepackage{mathrsfs}
\usepackage{graphicx}
\usepackage{amsfonts,delarray,amssymb,amsmath,amsthm,a4,a4wide}
\usepackage{latexsym}
\usepackage{epsfig}
\usepackage{color}
\usepackage{amsaddr}

\newtheorem{thm}{Theorem}[section]

\newtheorem{lem}[thm]{Lemma}
\newtheorem{prop}[thm]{Proposition}

\newtheorem{question}[thm]{Question}

\theoremstyle{remark}
\newtheorem{rem}[thm]{Remark}

\theoremstyle{definition}
\newtheorem{defn}[thm]{Definition}

\numberwithin{equation}{section}

\newcommand{\R}{\mathbb R}
\newcommand{\K}{\mathbb K}

\newcommand{\e}{\varepsilon}

\newcommand{\p}{\partial}

\newcommand{\comment}[1]{}
\def\h{\hspace*{.24in}}

\makeatletter
\@namedef{subjclassname@2020}{%
  \textup{2020} Mathematics Subject Classification}
\makeatother
\begin{document} 

\title[iterative scheme for the Monge--Ampere eigenvalue problem ]{Convergence of an iterative scheme for the Monge--Amp\`ere eigenvalue problem 
with general initial data
}
\author{Nam Q. Le}
\address{Department of Mathematics, Indiana University\\ 831 E 3rd St,
Bloomington, IN 47405, USA\\
Email address: nqle@iu.edu}
\thanks{The author was supported in part by the National Science Foundation under grant DMS-2452320.}



\begin{abstract}
 In this note, we revisit an iterative scheme, due to Abedin and Kitagawa (Inverse Iteration for the Monge--Amp\`ere Eigenvalue Problem, {\it Proc. Amer. Math. Soc.} {\bf 148} (2020), no. 11, 4875--4886), to solve 
 the Monge--Amp\`ere eigenvalue problem on a general bounded convex domain. Using a nonlinear integration by parts, we show that the scheme converges for all  convex initial data having finite and nonzero Rayleigh quotient to a nonzero Monge--Amp\`ere eigenfunction. As an application, we obtain an energy characterization of the Monge--Amp\`ere eigenfunctions.
\end{abstract}
\maketitle
\section{Introduction and  statement of the main result}
In this note, we revisit an iterative scheme, due to Abedin and Kitagawa  in their recent paper \cite{AK}, to solve 
 the Monge--Amp\`ere eigenvalue problem on a general bounded convex domain $\Omega$ in $\R^n$ $(n\geq 2)$:
 \begin{equation}
 \left\{
 \begin{alignedat}{2}
   \det D^{2} w~&=\lambda |w|^{n} \h~&&\text{in} ~\Omega, \\\
w &=0\h~&&\text{on}~\p \Omega.
 \end{alignedat}
 \right.
 \label{EVP_eq}
\end{equation}
\medskip
Before recalling relevant results, it is convenient to introduce some notation. 
Let
$$
\K = \{  w \in C(\overline{\Omega}):  
 ~w~\text{is convex, nonzero in } \Omega,~ w=0~\text{on}~\p\Omega \}.
$$
When $u$ is merely a convex function on $\Omega$, by an abuse of notation, we use $\det D^2 u\, dx$ to denote the Monge--Amp\`ere measure associated with $u$; see Section \ref{A_sec}.

\medskip
For a convex function $u$ on $\Omega$,  we define its Rayleigh quotient by
\begin{equation}
\label{RQ}
R(u) = \frac{\int_{\Omega} |u|\det D^2 u\,dx}{\int_{\Omega} |u|^{n+1}\,dx}.
\end{equation}
Implicit in the definition (\ref{RQ}) is the requirement that 
$\|u\|_{L^{n+1}(\Omega)}<\infty.$

\medskip
For general bounded convex domains $\Omega\subset\R^n$, the existence, uniqueness and variational characterization of the Monge--Amp\`ere eigenvalue, and uniqueness of  convex Monge--Amp\`ere eigenfunctions 
were obtained in \cite{L}. 
 They are
the singular counterparts of those established by  Lions \cite{Ls} and Tso \cite{Tso} in the smooth, uniformly convex setting. 
For the purpose of this note, we recall  here part of  \cite[Theorem 1.1]{L}.
\begin{thm}  (\cite{L})
\label{ev_thm}
 Let $\Omega$ be a bounded convex domain in $\R^n$. Define $\lambda=\lambda[\Omega]$ by
\begin{equation}
 \lambda[\Omega] =\inf_{w\in \K} R(w).
\label{lam_def}
 \end{equation}
 Then, the following facts hold.
 \begin{enumerate}
 \item[(i)] (Existence) The infimum in (\ref{lam_def}) is achieved by a 
 nonzero convex solution $w\in C^{0,\beta}(\overline{\Omega})\cap C^{\infty}(\Omega)$ for all $\beta\in (0, 1)$ to the eigenvalue problem (\ref{EVP_eq}).
 The constant $\lambda[\Omega]$ is called the Monge-Amp\`ere eigenvalue of $\Omega$ and $w$ is called a Monge--Amp\`ere eigenfunction of $\Omega$.
\item[(ii)] (Uniqueness) If the pair $(\Lambda, \tilde w)$ 
satisfies $\det D^2 \tilde w =\Lambda |\tilde w|^n$ in $\Omega$ where $\Lambda>0$ is a positive constant and 
$\tilde w\in \K$, then $\Lambda=\lambda[\Omega]$ and $\tilde w=m w$ for some positive constant $m$.
\end{enumerate}
\end{thm}
It was recently proved in \cite[Theorem 1.1]{L_AFST} that the Monge--Amp\`ere eigenfunctions of general bounded convex domains are in fact globally Lipschitz.

\medskip
In \cite{AK}, Abedin and Kitagawa introduce an iterative scheme
\begin{equation}\label{IIS}
 \left\{
\begin{alignedat}{2}
\det D^2u_{k+1} &= R(u_k) |u_k|^n &&\quad \text{in } \Omega, \\\
u_{k+1} &= 0 && \quad \text{on } \partial \Omega
\end{alignedat}
\right.
\end{equation}
 to solve 
 the Monge--Amp\`ere eigenvalue problem (\ref{EVP_eq}). Here $u_{k+1}$ is a convex Aleksandrov solution of (\ref{IIS}). We refer to Theorem \ref{Dir_thm} for the existence of $u_{k+1}$ and 
 to Definition \ref{Aleksol} for  the notion of Aleksandrov solutions to the Monge--Amp\`ere equation. 
 
 \medskip
 An interesting feature of the iterative scheme (\ref{IIS}) is that 
 the sequence $\{u_k\}_{k=0}^{\infty}$ is obtained by repeatedly inverting the Monge--Amp\`ere operator with Dirichlet boundary condition. 
 One notes that similar inverse iteration methods have been considered for the $p$-Laplace equation \cite{BEM, B, Er, HL}. 
  Abedin and Kitagawa establish the first inverse iteration result for the eigenvalue problem of a fully nonlinear degenerate elliptic equation for a large class of initial data.
 Their main convergence result states as follows.
 \begin{thm} (\cite[Theorem 1.4]{AK})
 \label{AKthm}
Let $\Omega \subset \R^n$ be a bounded convex domain. Let $u_0 \in C(\overline{\Omega})$ be a function satisfying for some constant $c_0>0$:
\begin{enumerate}\item[(i)] $u_0$ is convex and $u_0 \leq 0$ on $\partial \Omega$;
\item[(ii)] $R(u_0) < \infty$;
\item[(iii)] $\det D^2 u_0 \geq c_0$ in $\Omega$.
\end{enumerate}
For $k \geq 0$, define the sequence $u_k \in \K$ to be the solutions of the Dirichlet problem (\ref{IIS}).
Then $\{u_k\}$ converges uniformly on $\overline{\Omega}$ to a nonzero Monge--Amp\`ere eigenfunction $u_{\infty}$ of $\Omega$. 
Furthermore, 
\[\lim\limits_{k \to \infty} R(u_k)  = \lambda[\Omega].\]
\end{thm}
In \cite{AK}, the constant $c_0$ was taken to be $1$ but the proof works for all $c_0>0$. The conditions (i) and (iii) in Theorem \ref{AKthm} were used in \cite{AK} to show that, in the iterative scheme (\ref{IIS}), the sequence $\{u_k\}$ satisfies $u_k\leq \hat w$ for all $k\geq 0$ where $\hat w$ is a Monge--Amp\`ere eigenfunction of $\Omega$ with $\|\hat w\|_{L^{\infty}(\Omega)} = \left(c_0^{-1}\lambda[\Omega]\right)^{-1/n}$. This implies the lower bound $\|u_k\|_{L^{\infty}(\Omega)} \geq \left(c_0^{-1}\lambda[\Omega]\right)^{-1/n}$  which guarantees the nontriviality of the limit
$u_{\infty}$ of $u_k$. Here we call a function $w$ trivial if $w\equiv 0$ in $\Omega$.

\begin{rem} Clearly  (i) and (iii) in Theorem \ref{AKthm} imply that $R(u_0)>0$. 
Without (i) and (iii) in Theorem \ref{AKthm}, $R(u_0)$ can be $0$  and thus the scheme (\ref{IIS}) gives $u_k\equiv 0$ for all $k\geq 1$.  For example, if $u_0$ is a nonzero affine function, then $R(u_0)=0$.
Thus, to get the nontriviality of the limit
$u_{\infty}$ of $u_k$, if it exists, we need to require that $R(u_0)$ be nonzero. 
\end{rem}
In this note, we remove the restrictions (i) and (iii) in Theorem \ref{AKthm}. We show that the iterative scheme (\ref{IIS})  converges for all  convex initial data having finite and nonzero Rayleigh quotient
to a nonzero Monge--Amp\`ere eigenfunction of $\Omega$. Thus our result covers
 all possible convex functions on $\Omega$ as initial data for the scheme (\ref{IIS}). 
 \begin{thm}\label{IISthm}
Let $\Omega \subset \R^n$ ($n\geq 2$) be a bounded convex domain. Let $u_0 \in C(\Omega)$ be a nonzero convex function on $\Omega$ with $0<R(u_0)<\infty$.
For $k \geq 0$, define the sequence $u_k \in \K$ to be the solutions of the Dirichlet problem (\ref{IIS}).
Then $\{u_k\}$ converges uniformly on $\overline{\Omega}$ to a nonzero Monge--Amp\`ere eigenfunction $u_{\infty}$ of $\Omega$. 
Furthermore, \[\lim\limits_{k \to \infty} R(u_k)  = \lambda[\Omega],\]
and for all $k\geq 3$ and any fixed nonzero Monge--Amp\`ere eigenfunction $w$, we have
\begin{equation*} \begin{split} [R(u_k)]^{1/n}-(\lambda[\Omega])^{1/n} &\leq  (\lambda[\Omega])^{1/n}\frac{  \int_{\Omega}(|u_{k+1}|- |u_{k}|) |w|^n~dx} {\int_{\Omega} |u_3| |w|^n~dx}\\&\leq C(u_0, \Omega, n)  \int_{\Omega}|u_\infty- u_{k}| ~dx.\end{split}\end{equation*}
\end{thm}
 The nontriviality of our limit under the nonzero finiteness of the Rayleigh quotient $R(u_0)$ is due to an eventual regularity of the scheme (Proposition \ref{reg_prop}) and an important monotonicity formula during the scheme (Lemma \ref{mono_lem}).  The proof of this monotonicity formula is based on a {\it nonlinear integration by parts}, established in \cite{L}, which was designed to prove uniqueness results for the Monge--Amp\`ere equations and systems of Monge--Amp\`ere equations \cite{L, L_MAA}.
 \begin{rem}
 Theorem \ref{IISthm} also bounds the convergence rate of $R(u_k)$ to the Monge--Amp\`ere eigenvalue $\lambda[\Omega]$ in terms of the convergence rate of $u_k$ to the nonzero Monge--Amp\`ere eigenfunction $u_{\infty}$. 
 Compared to the inverse iteration methods for the $p$-Laplace equation in \cite{BEM, B, HL}, this type of estimate seems to be new.
\end{rem}

\begin{rem}
For any nonzero convex function $u_0 \in C(\Omega)$ with $0<R(u_0)<\infty$, the scheme (\ref{IIS}) gives, for all $k\geq 1$,  $u_k\in C(\overline{\Omega})$ with $u_k=0$ on $\p\Omega$ and $0<R(u_k)<\infty$. In other words, conditions (i) and (ii) of Theorem \ref{AKthm} are satisfied for the scheme $u_k$ as initial data. However, the convergence result in Theorem \ref{IISthm} cannot be deduced from Theorem \ref{AKthm} as the condition (iii) there may not hold for all $k\geq 1$. 
For example, we can take any  $u_0 \in C(\Omega)$ with $0<R(u_0)<\infty$ and $u_0(z)=0$ for some $z\in\Omega$. Then $u_1$ does not satisfy (iii) of Theorem \ref{AKthm}. Since $u_k=0$ on $\p\Omega$ for all $k\geq 1$, from the first equation of (\ref{IIS}), we see that $u_{k+1}$ does not  satisfy (iii) either.
\end{rem}

\begin{rem}
The Abedin--Kitagawa iterative scheme has been numerically implemented by Liu--Leung--Qian \cite{LLQ}.
\end{rem}

As an application of Theorem \ref{IISthm} and a Monge--Amp\`ere Schwarz inequality, we obtain an energy characterization of the Monge--Amp\`ere eigenfunctions. This is the content of the next theorem.
\begin{thm}
\label{lamC}
Let $\lambda[\Omega]$ be the Monge--Amp\`ere eigenvalue of a bounded convex domain $\Omega$ in $\R^n$. 
  Let $u\in C(\overline{\Omega})$ be a nonzero convex function on $\Omega$ with $u=0$ on $\p\Omega$ that satisfies
\[\int_\Omega |u|\det D^2 u\, dx =\lambda[\Omega] \int_\Omega |u|^{n+1}\, dx.\]
Then $u$ is a Monge--Amp\`ere eigenfunction of $\Omega$, that is, 
\[\det D^2 u=\lambda[\Omega] |u|^n\quad \text{in }\Omega,\quad \text{and } u=0\quad\text{on }\p\Omega.\]
\end{thm}

The rest of this note is organized as follows. In Section \ref{A_sec}, we recall basic facts on the  Monge--Amp\`ere equation and prove a reverse Aleksandrov estimate in Proposition \ref{ReA}. 
In Section \ref{Eventual_sec}, we show the eventual smoothness and a new monotonicity formula for the iterative scheme (\ref{IIS}). The proof of Theorem \ref{IISthm}  will be given in Section
\ref{pf_sec}. In Section \ref{EVP_pf}, we make some remarks on the energy characterization of the Monge--Amp\`ere eigenfunctions and prove Theorem \ref{lamC}.

\section{The Monge--Amp\`ere equation  and a reverse Aleksandrov estimate}
\label{A_sec}
Here, we recall some basic facts on the Monge--Amp\`ere equation on convex domains $\Omega$ of $\R^{n}$ $(n\geq 2)$; see the books by Figalli \cite{Fi} and Guti\'errez \cite{G} for more details.
We will establish a reverse Aleksandrov estimate in Proposition \ref{ReA} that could be of independent interest. 

 For a convex function \(u:\Omega \to \R\), we define
 the subdifferential of $u$  at $x\in\Omega$ by
 $$
\partial u (x):=\{p\in \R^{n}\,:\, u(y)\ge u(x)+p\cdot (y-x)\quad \text{for all } y \in \Omega\}.
$$
Below is a precise definition of the Monge--Amp\`ere measure of a convex function  \(u:\Omega \to \R\); see also \cite[Definition 2.1]{Fi} and \cite[Theorem 1.1.13]{G}.
\begin{defn}[Monge--Amp\`ere measure]
\label{MAdef}
Let $u:\Omega\rightarrow \R$ be a convex function. The Monge--Amp\`ere measure, $Mu$, associated with the convex function $u$ is defined by
$$Mu(E) = |\p u(E)|~\text{where } \p u(E) = \bigcup_{x\in E} \p u(x),~\text{for each Borel set } E\subset\Omega.$$
If $u\in C^{2}(\Omega)$, then 
$
Mu=\det D^{2} u(x)\,dx$ in $\Omega$.
\end{defn}

\begin{defn}[Aleksandrov solutions]
\label{Aleksol}Given a convex domain \(\Omega\) and a Borel measure \(\mu\) on \(\Omega\), we call a convex
function \(u:\Omega \to \R\) an \emph{Aleksandrov solution} to the Monge--Amp\`ere equation
$$
\det D^{2} u =\mu,
$$
if $\mu=Mu$ as Borel measures. When \(\mu=f\,dx\), we will say for simplicity  that \(u\) solves 
\begin{equation*}
\det D^2 u=f.
\end{equation*}
\end{defn}

In this note, we use $\det D^2 u$ to denote the Monge--Amp\`ere measure $Mu$ for a general convex function $u$. Thus, for all Borel sets 
$E\subset\Omega$,
$$\int_E |u|\det D^2 u\,dx=\int_E |u| dMu.$$

Now, we recall the basic existence and uniqueness result for solutions to the Dirichlet problem with zero boundary data for the Monge--Amp\`ere equation;  see  \cite[Theorem 2.13]{Fi}, \cite[Theorem 1.6.2]{G}, and
\cite[Theorem 1]{Har1}.
\begin{thm}[The Dirichlet problem] 
\label{Dir_thm}
Let $\Omega$ be a bounded convex domain in $\R^n$, and let $\mu$ be a nonnegative Borel measure in $\Omega$.
Then there exists a unique convex function $u\in C(\overline{\Omega})$ that is an Aleksandrov solution of
\begin{equation*}
 \left\{
 \begin{alignedat}{2}
   \det D^{2} u~&=\mu \h~&&\text{in} ~\Omega, \\\
u &=0\h~&&\text{on}~\p \Omega.
 \end{alignedat}
 \right.
\end{equation*}
\end{thm}

For later reference, we state the celebrated Aleksandrov maximum principle for the Monge--Amp\`ere equation; see \cite[Theorem 2.8]{Fi} and \cite[Theorem 1.4.2]{G}.

\begin{thm}[Aleksandrov's maximum principle]\label{Alek_thm} Let  $\Omega\subset\R^n$ be a bounded convex domain. Let $u\in C(\overline{\Omega})$ be a convex function.
If $u=0$ on $\p \Omega$, then
\[
|u(x)|^{n}\le C(n)(\emph{diam}\Omega)^{n-1}\emph{dist}(x,\partial \Omega)\int_{\Omega}\det D^2 u\,dx\qquad \text{ for all } x\in\Omega.
\]
 \end{thm}
 
 We have the following proposition which will play a crucial role in the proof of Theorem \ref{IISthm}.
\begin{prop}[Reverse Aleksandrov estimate]
\label{ReA}
 Let $\Omega$ be a bounded convex domain in $\R^n$.  Let $\lambda[\Omega]$ be the Monge--Amp\`ere eigenvalue of $\Omega$ and let $w$ be a nonzero Monge--Amp\`ere eigenfunction of $\Omega$.
Assume that $u\in C^{5}(\Omega)\cap C(\overline{\Omega})$ is a strictly convex function in $\Omega$ with  $u=0$ on $\Omega$ and satisfies
$$\int_\Omega (\det D^2 u)^{1/n} |w|^{n-1}\,dx<\infty.$$
Then
\begin{equation}
\label{ReA2}
\int_\Omega (\lambda[\Omega])^{1/n} |u| |w|^n\,dx\geq \int_\Omega (\det D^2 u)^{1/n} |w|^n\,dx.
\end{equation}
\end{prop}
\begin{rem}
Compared to Theorem \ref{Alek_thm}, the function $u$ appears on the dominating side in (\ref{ReA2}) in Proposition \ref{ReA}. For this reason,
(\ref{ReA2}) can be viewed as a sort of reverse Aleksandrov estimate. Moreover, this estimate is sharp. When $u$ is a Monge--Amp\`ere eigenfunction of $\Omega$, (\ref{ReA2}) is an equality.
\end{rem}
To prove Proposition \ref{ReA}, we recall the following {\it nonlinear integration by parts} established in \cite[Proposition 1.7]{L}.
\begin{prop}[Nonlinear integration by parts]
 \label{NIBP} Let $\Omega$ be a bounded convex domain in $\R^n$.
 Let $u, v\in C(\overline{\Omega})\cap C^5 (\Omega)$ be strictly convex functions in $\Omega$ with $u=v=0$ on $\p\Omega$. If 
 \begin{equation*}\int_{\Omega}(\det D^2 u)^{\frac{1}{n}}  (\det D^2 v)^{\frac{n-1}{n}}\,dx<\infty,~\text{and}~\int_{\Omega}\det D^2 v\,dx<\infty,
 \end{equation*} then
\begin{equation*} \int_{\Omega} |u|\det D^2 v\,dx \geq \int_{\Omega} |v|(\det D^2 u)^{\frac{1}{n}} (\det D^2 v)^{\frac{n-1}{n}}\,dx.
\end{equation*}
\end{prop}
\begin{proof}[Proof of Proposition \ref{ReA}]
We apply Proposition \ref{NIBP} to $u$ and $v=w$. 
Then, using $\det D^2 w= \lambda[\Omega] |w|^n$, we get
\begin{eqnarray*}\int_\Omega  |u| \lambda[\Omega] |w|^n\, dx= \int_\Omega  |u| \det D^2 w\,dx&\geq& \int_\Omega |w|(\det D^2 u)^{1/n} (\det D^2 w)^{\frac{n-1}{n}}\,dx
\\&=&  \int_\Omega ( \lambda[\Omega])^{\frac{n-1}{n}} (\det D^2 u)^{1/n} |w|^n\,dx.
\end{eqnarray*}
Dividing the first and last expressions in the above estimates by $  ( \lambda[\Omega])^{\frac{n-1}{n}}$, we obtain (\ref{ReA2}).
\end{proof}

\section{Eventual smoothness and a new monotonicity formula for the iterative scheme}
\label{Eventual_sec}
In this section, we show the eventual smoothness  and a new monotonicity formula for the iterative scheme (\ref{IIS}).  They are stated in Proposition \ref{reg_prop} and Lemma \ref{mono_lem}.

We have the following eventual smoothness of solutions to the iterative scheme (\ref{IIS}). The proof is a modification of the proof of Proposition 2.8 in  \cite{L}.
\begin{prop}
\label{reg_prop}
 Let $\Omega$ be a bounded convex domain in $\R^n$ with nonempty interior. Let $u_0 \in C(\Omega)$ be a nonzero convex function on $\Omega$ with $0<R(u_0)<\infty$.
For $k \geq 0$, define the sequence $u_k \in \K$ to be the solutions of the Dirichlet problem (\ref{IIS}).
 Then, $u_1\in C^{0, \frac{1}{n}}(\Omega)$, and  $u_{k+1}$ is strictly convex in $\Omega$ and $u_{k+1}\in C^{2k, \frac{1}{n}}(\Omega)$ for all $k\geq 1$. 
\end{prop}
We recall that a convex function $u$ on a bounded convex domain $\Omega$ is said to be strictly convex in $\Omega$, if for any $x\in\Omega$ and $p\in\p u(x)$, 
$$u(z)> u(x) + p\cdot (z-x)~\text{for all~} z\in\Omega\backslash \{x\}.$$

\begin{proof} [Proof of Proposition \ref{reg_prop}]
First, using the Aleksandrov maximum principle in Theorem \ref{Alek_thm}, we note that each $u_{k+1}$ is uniformly bounded, that is $M_{k+1}=\|u_{k+1}\|_{L^{\infty}(\Omega)}<\infty.$ 
The regularity $C^{0, \frac{1}{n}}(\overline{\Omega})$ of $u_1$ is a consequence of the Aleksandrov maximum principle. Since $u_0\not\equiv 0$, and $R(u_0)>0$,  we have 
$u_1\not\equiv 0$. The convexity of $u_1$ shows that $u_1<0$ in $\Omega$.

  \medskip
We show by induction the following:\\
{\bf Claim. } $u_{k+1}$ is strictly convex in $\Omega$ and $u_{k+1}\in C^{2k, \frac{1}{n}}(\Omega)$ for all $k\geq 1$. 

We start with the base case $k=1$. For each $\e\in (0, M_2)$, let $\Omega':=\Omega(\e)=\{x\in\Omega: u_{2}(x)<-\e\}$. Since $u_{2}\in C(\overline{\Omega})$ is convex, the set $\Omega(\e)$ is convex
with nonempty interior. 
Note that, since $|u_1|>0$ in $\Omega$ and $u_1\in C^{0, \frac{1}{n}}(\overline{\Omega})$, by continuity, $|u_1|\geq m(n, \e)>0$ in $\overline{\Omega'}$.
Since $$\lambda [\Omega]m^n(n,\e)\leq \det D^2 u_{2}=R(u_1)|u_1|^n  \leq R(u_1) M_1^n \text{ in }\Omega' \text{ and }u_{2}=-\e \text{ on }\p\Omega', $$ the function $u_{2}$ is strictly convex in $\Omega'$ by the
localization theorem of Caffarelli \cite{C1} (see also \cite[Theorem 4.10]{Fi} and \cite[Corollary 5.2.2]{G}). Moreover,  $u_1\in C^{0,\frac{1}{n}}(\Omega')$. Now, using Caffarelli's $C^{2,\alpha}$ estimates \cite{C2}, we have
$u_2\in C^{2,\frac{1}{n}}_{\text{loc}}(\Omega')$. Since $\e\in (0, M_2)$ is arbitrary, we conclude $u_2\in C^{2,\frac{1}{n}}(\Omega)$ and $u_2$ is strictly convex in $\Omega$.
\vglue 0.2cm
Suppose the claim holds up to $k-1$ where $k\geq 2$. We show it also holds for $k$. For each $\e_k\in (0, M_{k+1})$, let $\Omega(\e_k)=\{x\in\Omega: u_{k+1}(x)< -\e_k\}$. Since $u_{k+1}\in C(\overline{\Omega})$ is convex, the set $\Omega(\e_k)$ is convex
with nonempty interior. Let us denote $\Omega_k'=\Omega(\e_k)$ for brevity. Note that, by continuity, $|u_k|\geq m(n, k, \e)>0$ in $\overline{\Omega_k'}$.

\medskip
Since $\lambda[\Omega] m^n(n, k,\e)\leq \det D^2 u_{k+1}=R(u_k)|u_k|^n  \leq R(u_k) M_k^n$ in $\Omega_k'$ and $u_{k+1}=-\e_k$ on $\p\Omega_k'$, the function $u_{k+1}$ is strictly convex in $\Omega_k'$ by the
localization theorem of Caffarelli. By the induction hypothesis, $u_k\in C^{2(k-1),\frac{1}{n}}(\Omega_k')$.
In the interior of $\Omega_k'$, the equation $\det D^2 u_{k+1}= R(u_k)|u_k|^n$ now becomes uniformly elliptic with $C^{2(k-1),\frac{1}{n}}$ right hand side. Therefore, 
 we
have $u_{k+1}\in C^{2k,\frac{1}{n}}_{\text{loc}}(\Omega_k')$. Since $\e_k\in (0, M_{k+1})$ is arbitrary, we conclude $u_{k+1}\in C^{2k, \frac{1}{n}}(\Omega)$ and $u_{k+1}$ is strictly convex in $\Omega$.
\end{proof}
Our key observation is the following monotonicity result for the iterative scheme (\ref{IIS}). Note that $$R(u_k)\geq \lambda[\Omega] \quad \text{for all }k\geq 1.$$
\begin{lem} [Monotonicity formula for the iterative scheme]
\label{mono_lem}
 Let $\Omega$ be a bounded convex domain in $\R^n$. 
 Let $u_0 \in C(\Omega)$ be a nonzero convex function on $\Omega$ with $0<R(u_0)<\infty$.
 Let $w$ be a nonzero Monge--Amp\`ere eigenfunction of $\Omega$. Consider the iterative scheme (\ref{IIS}).
If $k\geq 3$, then
$$\int_{\Omega} |u_{k+1}| |w|^n\,dx\geq \int_{\Omega} |u_k| |w|^n\,dx + \frac{[R(u_k)]^{1/n}-(\lambda[\Omega])^{1/n}}{(\lambda[\Omega])^{1/n}} \int_{\Omega} |u_k| |w|^n\,dx .$$
\end{lem}
\begin{proof}[Proof of Lemma \ref{mono_lem}]
By Proposition \ref{reg_prop}, we have $u_{k+1}\in C^{6, \frac{1}{n}}(\Omega)$ for all $k\geq 3$. We apply Proposition \ref{ReA} to $u_{k+1}$ and recall
$$\det D^2 u_{k+1}= R(u_k)|u_{k}|^n,$$
 to get
\begin{eqnarray*}\int_\Omega  |u_{k+1}| |w|^n\,dx&\geq& \frac{1}{(\lambda[\Omega])^{1/n}} \int_\Omega (\det D^2 u_{k+1})^{1/n} |w|^n\,dx
\\ &=& \frac{[R(u_k)]^{\frac{1}{n}}}{(\lambda[\Omega])^{1/n}} \int_\Omega |u_k| |w|^n\,dx\\
&=&  \int_{\Omega} |u_k| |w|^n\,dx + \frac{[R(u_k)]^{1/n}-(\lambda[\Omega])^{1/n}}{(\lambda[\Omega])^{1/n}} \int_{\Omega} |u_k| |w|^n\,dx.
\end{eqnarray*}
The monotonicity property is thus proved.
\end{proof}
We recall the following monotonicity property in \cite[Lemma 3.1]{AK}.
\begin{lem} (\cite[Lemma 3.1]{AK})
\label{AKR}
 Let $\Omega$ be a bounded convex domain in $\R^n$ with nonempty interior. Let $u_0 \in C(\Omega)$ be a nonzero convex function on $\Omega$ with $0<R(u_0)<\infty$.
Consider the iterative scheme (\ref{IIS}). Then for all $k\geq 0$, we have
$$R(u_{k+1}) \|u_{k+1}\|^n_{L^{n+1}(\Omega)}\leq  R(u_k) \|u_k\|^n_{L^{n+1}(\Omega)}.$$
\end{lem} 
Lemma \ref{AKR} was stated and proved in \cite{AK} for $u_0$ satisfying (i), (ii) and (iii) in Theorem \ref{AKthm}. However, the proof in \cite{AK} only used the assumptions $0<R(u_0)<\infty$ and $u_0$ is convex.
We include here the short proof of Lemma \ref{AKR} for reader's convenience.
\begin{proof}[Proof of Lemma \ref{AKR}]
The proof follows by multiplying both sides of the first equation of (\ref{IIS}) by $|u_{k+1}|$, integrating over $\Omega$ and then using the H\"older inequality:
\begin{eqnarray*}R(u_{k+1}) \|u_{k+1}\|^{n+1}_{L^{n+1}(\Omega)}&=&\int_{\Omega} |u_{k+1}|\det D^2 u_{k+1}\,dx\\& =&R(u_k)\int_{\Omega} |u_k|^n |u_{k+1}|\leq  R(u_k) \|u_k\|^n_{L^{n+1}(\Omega)} \|u_{k+1}\|_{L^{n+1}(\Omega)}.
\end{eqnarray*}
Using $u_{k+1}\not \equiv 0$ for all $k\geq 0$, we obtain the claimed monotonicity property.
\end{proof}

\section{Convergence of the iterative scheme}
\label{pf_sec}
In this section, we prove Theorem \ref{IISthm}.
\vglue 0.2cm
Some of our arguments in {\it Step 2} of the proof of Theorem \ref{IISthm} are similar to those in the proof of Theorem \ref{AKthm} in \cite{AK}. However, since we can obtain the convergence of $R(u_k)$ to
$\lambda[\Omega]$ from Lemma \ref{mono_lem}, we can avoid using the continuity property of the energy $\int_{\Omega} |u_k|\det D^2 u_k \,dx$ for a converging sequence of convex functions $u_k$ with an upper bound on the density of the Monge--Amp\`ere measure $\det D^2 u_k$ (see \cite[Lemma 2.9]{AK}). Moreover, the monotone property in Lemma \ref{mono_lem} also allows us to quickly conclude that the whole sequence $u_k$ converges to the same limit.
\begin{proof}[Proof of Theorem \ref{IISthm}]

We fix a nonzero Monge--Amp\`ere eigenfunction $w$. The assumptions on $u_0$ imply that $u_k\not\equiv 0$ for all $k\geq 0$.
The proof is split into several steps.
\vglue 0.1cm
\noindent
{\it Step 1: The whole sequence $R(u_{k})$ converges to $\lambda[\Omega]$.}

Using the monotonicity property established in Lemma \ref{mono_lem}, we find that if $k\geq 3$, then
\begin{equation}
\label{uknot0}
\|u_k\|_{L^{\infty}(\Omega)} \geq \frac{\int_\Omega |u_k| |w|^n\,dx}{\int_\Omega |w|^n\,dx}\geq  \frac{\int_\Omega |u_3| |w|^n\,dx}{\int_\Omega |w|^n\,dx}\geq c(n,\Omega, u_0)>0.
\end{equation}
For each $k\geq 1$, using $R(u_k)\geq \lambda[\Omega]$, we obtain from Lemma \ref{AKR} that
\begin{equation}
\label{ukbound}
\|u_k\|^n_{L^{n+1}(\Omega)}\leq \frac{R(u_0)\|u_0\|^n_{L^{n+1}(\Omega)}}{ R(u_{k})} \leq \frac{R(u_0)\|u_0\|^n_{L^{n+1}(\Omega)}}{ \lambda[\Omega]}<\infty.
\end{equation}
This implies that the increasing sequence $\int_{\Omega}|u_k| |w|^n\,dx$ is bounded from above and thus converges to a limit $L$
\begin{equation}
\label{Llim}
\lim_{k\rightarrow \infty} \int_\Omega |u_k||w|^n\, dx=L\in (0, \infty),
\end{equation}
where we used (\ref{uknot0}) to get $L>0$.

Now, for $k\geq 3$, taking into account  the full monotonicity property in Lemma \ref{mono_lem}, 
we get
\begin{eqnarray}
\label{Ruk}
 [R(u_k)]^{1/n}-(\lambda[\Omega])^{1/n} &\leq& (\lambda[\Omega])^{1/n}\frac{  \int_{\Omega}(|u_{k+1}|- |u_{k}|) |w|^n\,dx} {\int_{\Omega} |u_k| |w|^n\,dx}\nonumber\\
&\leq&  (\lambda[\Omega])^{1/n}\frac{  \int_{\Omega}(|u_{k+1}|- |u_{k}|) |w|^n\,dx} {\int_{\Omega} |u_3| |w|^n\,dx}.
\end{eqnarray}
Letting $k\rightarrow\infty$  in (\ref{Ruk}) and recalling (\ref{Llim}), we conclude that the whole sequence $R(u_{k})$ converges to $\lambda[\Omega]$:
\begin{equation}
\label{Ru_con}
\lim_{k\rightarrow\infty} R(u_k)= \lambda[\Omega].
\end{equation}
\noindent
{\it Step 2: Convergence of $u_k$ to a nontrivial Monge--Amp\`ere eigenfunction $u_{\infty}$ of $\Omega$.}

Next, applying the Aleksandrov estimate in Theorem \ref{Alek_thm} to $u_{k+1}$ where $k\geq 0$, and then using the H\"older inequality together with (\ref{ukbound}), we find
\begin{eqnarray*}
\|u_{k+1}\|^n_{L^{\infty}(\Omega)} \leq C(n,\Omega)\int_{\Omega}\det D^2 u_{k+1}~ dx &=& C(n,\Omega) R(u_k)\int_{\Omega} |u_k|^n~ dx \\& \leq&  C(n,\Omega) R(u_k) 
 \|u_k\|^n_{L^{n+1}(\Omega)}|\Omega|^{\frac{1}{n+1}}\\
& \leq& C(n,\Omega, u_0).
\end{eqnarray*} 
Hence, we obtain the uniform $L^{\infty}$ bound
$$\|u_k\|_{L^{\infty}(\Omega)} \leq C(n,\Omega, u_0)<\infty.$$
From the Aleksandrov estimate, we have the uniform $C^{0, \frac{1}{n}} (\overline{\Omega})$ bound for $u_k$ when $k\geq 1$:
$$\|u_k\|_{C^{0, \frac{1}{n}} (\overline{\Omega})} \leq C(n,\Omega, u_0).$$ 
Therefore, up to extracting a subsequence, we have the following uniform convergence $$u_{k_j}\rightarrow u_{\infty}\not \equiv 0$$
for a convex function $u_{\infty}\in C(\overline{\Omega})$ with $u_{\infty}=0$ on $\p\Omega$ while we also have the uniform convergence 
 $$u_{k_j + 1}\rightarrow w_{\infty}\not\equiv 0$$ 
 for a convex function $w_{\infty}\in C(\overline{\Omega})$ with $w_{\infty}=0$ on $\p\Omega$.
 
Thus, letting $j\rightarrow\infty $ in $$\det D^2 u_{k_{j} +1}=R(u_{k_j})|u_{k_j}|^n,$$ using (\ref{Ru_con}) and the weak convergence of the Monge--Amp\`ere measure (see \cite[Corollary 2.12]{Fi} and \cite[Lemma 5.3.1]{G}), we get
\begin{equation}
\label{uwinfi}
\det D^2 w_{\infty}= \lambda[\Omega] |u_{\infty}|^n.
\end{equation}
Letting $j\rightarrow \infty$ in the following monotonicity property (see Lemma \ref{AKR})
$$ R(u_{k_{j+1}}) \|u_{k_{j+1}}\|^n_{L^{n+1}(\Omega)}\leq R(u_{k_j+1}) \|u_{k_j+1}\|^n_{L^{n+1}(\Omega)}\leq  R(u_{k_j}) \|u_{k_j}\|^n_{L^{n+1}(\Omega)},  $$
and recalling (\ref{Ru_con}), we find that
$$\|w_{\infty}\|_{L^{n+1}(\Omega)} = \|u_{\infty}\|_{L^{n+1}(\Omega)}. $$
However, from (\ref{uwinfi}), we have
\begin{eqnarray*}R(w_\infty) \|w_{\infty}\|^{n+1}_{L^{n+1}(\Omega)}=\int_{\Omega} |w_{\infty}| \det D^2 w_{\infty}\,dx &=&\lambda[\Omega] \int_{\Omega} |u_{\infty}|^n |w_{\infty}| \,dx\\ &\leq& 
\lambda[\Omega]  \|u_{\infty}\|^n_{L^{n+1}(\Omega)}  \|w_{\infty}\|_{L^{n+1}(\Omega)}\\&=& \lambda[\Omega]  \|w_{\infty}\|^{n+1}_{L^{n+1}(\Omega)}.
\end{eqnarray*}
Since $R(w_{\infty})\geq \lambda[\Omega]$,  we must have $R(w_{\infty})= \lambda[\Omega]$, and the inequality above must be an equality,  
but this gives $u_{\infty}= c w_{\infty}$ for some constant $c>0$. Thus, from (\ref{uwinfi}), we have $$\det D^2 w_{\infty} = c^n\lambda[\Omega] |w_\infty|^n.$$ It follows from the uniqueness part of Theorem \ref{ev_thm} that $c=1$ and  $w_\infty= u_{\infty}$ is a Monge--Amp\`ere eigenfunction of $\Omega$.
Passing to the limit in Lemma \ref{mono_lem}, we have
$$\int_\Omega |u_{\infty}| |w|^n \, dx=\lim_{k\rightarrow \infty} \int_\Omega |u_k||w|^n\, dx=L.$$
With this property and the uniqueness up to positive multiplicative constants of the Monge--Amp\`ere eigenfunctions of $\Omega$, we conclude that the limit $u_{\infty}$ does not depend on the subsequence $u_{k_j}$. This shows that the whole sequence $u_k$ converges to a nonzero Monge--Amp\`ere eigenfunction $u_{\infty}$ of $\Omega$. 
\vglue 0.1cm
\noindent
{\it Step 3: Convergence estimate for  $[R(u_k)]^{1/n}-(\lambda[\Omega])^{1/n}$. }

Let $k\geq 3$.
By (\ref{Ruk}), and the fact that $\int_\Omega |u_k| |w|^n\,dx$ increases to $L=\int_{\Omega} |u_{\infty}| |w|^n\,dx$, we have the estimates
\begin{eqnarray*} [R(u_k)]^{1/n}-(\lambda[\Omega])^{1/n} &\leq& (\lambda[\Omega])^{1/n}\frac{  \int_{\Omega}(|u_{k+1}|- |u_{k}|) |w|^n~dx} {\int_{\Omega} |u_3| |w|^n\,dx}\\
&\leq&  (\lambda[\Omega])^{1/n}\frac{  \int_{\Omega}(|u_\infty|- |u_{k}|) |w|^n\,dx} {\int_{\Omega} |u_3| |w|^n\,dx}\\
&\leq& C(n,\Omega, u_0)\int_{\Omega}|u_{\infty}- u_k|\,dx.
\end{eqnarray*}
The last statement of the theorem follows.
\end{proof}
\begin{rem} The inverse iterative scheme \eqref{IIS} was extended in \cite{L_RMI} to the $k$-Hessian eigenvalue problem on smooth, bounded $(k-1)$-convex domains in $\R^n$ where $1\leq k\leq n$. Except for the cases $k=1$  (the Laplace eigenvalue problem) and $k=n$ (the Monge--Amp\`ere eigenvalue problem), many issues are left open when $2\leq k\leq n-1$.
When $2\leq k\leq n-1$, only convergence (with rate) for the $k$-Hessian eigenvalue was proved in \cite{L_RMI}, and as explained in \cite[Remark 5.3]{L_RMI},  it remains an interesting open problem to prove the convergence of the scheme to the $k$-Hessian eigenfunction. 
\end{rem}
\section{Energy characterization of Monge--Amp\`ere eigenfunctions}
\label{EVP_pf}
In this section, we make some remarks on the energy characterization of  the Monge--Amp\`ere eigenfunctions motivated from the  proof of Theorem \ref{IISthm} in Section \ref{pf_sec} and prove Theorem \ref{lamC}.

\medskip
Observe that the Monge--Amp\`ere measure of each $u_{k_j+1}$ has density $R(u_{k_j}) |u_{k_j}|^n$ which is uniformly bounded from above by a positive constant independent of $k$. Thus, using  the continuity property of the energies $\int_{\Omega} |u_{k_j+1}|\det D^2 u_{k_j+1}\,dx$ (see \cite[Lemma 2.9]{AK}), 
we get
$$\lim_{j\rightarrow\infty} \int_{\Omega} |u_{k_j+1}|\det D^2 u_{k_j+1} \,dx= \int_{\Omega} |w_{\infty}|\det D^2 w_{\infty}\,dx$$
so, by (\ref{Ru_con}) 
\begin{equation}
\label{Rlam}
R(w_{\infty})= \lim_{j\rightarrow\infty}  R(u_{k_j +1})= \lambda[\Omega].
\end{equation}
We would like to show that $w_{\infty}$ is a Monge--Amp\`ere eigenfunction of $\Omega$. In the proof of Theorem \ref{IISthm}, we used the monotonicity property of the scheme (\ref{IIS}) given by Lemma \ref{AKR}. 
Finding a direct proof from (\ref{Rlam}) leads us to the following question:
\begin{question}
\label{EVPQ}
Assume that $u\in \K$ satisfies $R(u) =\lambda[\Omega]$. Is $u$  a Monge--Amp\`ere eigenfunction?
\end{question}
It is well known that for a bounded domain in $\Omega\subset\R^n$, if $v\in W^{1, 2}_0(\Omega)\backslash \{0\}$ satisfies
$$\int_{\Omega} |Dv|^2 \,dx =\lambda_1 \int_{\Omega}|v|^2\,dx,$$
where $\lambda_1$ is the first eigenvalue of the Laplace operator with zero boundary condition in $\Omega$, then $v$ is in fact a first eigenfunction of the Laplace operator on $\Omega$.  

\medskip
An affirmative answer to Question \ref{EVPQ} will provide a nonlinear analogue of the above result. 

\medskip
We could not find a direct answer to Question \ref{EVPQ}. As mentioned by one of the reviewers, a positive answer to a complex analogue of Question \ref{EVPQ} concerning the complex
Monge--Amp\`ere eigenvalue problem was given by Badiane--Zeriahi \cite{BZ1, BZ2} using plurisubharmonic envelopes.

\medskip

Interestingly, we can positively answer Question \ref{EVPQ} in Theorem \ref{lamC} by using the convergence result in Theorem \ref{IISthm} and a nonlinear Schwarz inequality.  The rest of this section is devoted to the proof of Theorem \ref{lamC}.

\medskip

In \cite[Theorem 3.1]{V}, 
Verbitsky proved a Monge--Amp\`ere Schwarz inequality for $C^2$ convex functions vanishing at the boundary of smooth and uniformly convex domains.
We will use the following extension, due to Huang \cite[Lemma 4.1]{H}, for convex functions vanishing at the boundary of general bounded convex domains.

\begin{lem}[Monge--Amp\`ere Schwarz inequality]
\label{NCS}
 Let $\Omega$ be a bounded convex domain in $\R^n$ and let $u, v\in C(\overline{\Omega})$ be convex functions on $\Omega$ with $u=v=0$ on $\p\Omega$. Then
\begin{equation*}
\int_\Omega|u| \det D^2 v\, dx \leq \Big(\int_\Omega|u| \det D^2 u\, dx\Big)^{\frac{1}{n+1}} \Big(\int_\Omega|v| \det D^2 v\, dx\Big)^{\frac{n}{n+1}}.
\end{equation*}
\end{lem}

For a convex function $u\in C(\overline{\Omega})$ with $u=0$ on $\p\Omega$, we denote its Monge--Amp\`ere energy by
\[E(u):=\int_\Omega (-u)\det D^2 u\, dx\equiv \int_\Omega |u|\det D^2 u\, dx.\]
\medskip

Using the Monge--Amp\`ere Schwarz inequality, we obtain further monotonicity properties of the scheme \eqref{IIS} when the initial data vanishes on the boundary; see also Lu--Zeriahi \cite[Lemma 5.1]{LZ} and Zeriahi \cite[Lemma 3.4]{Z}
for the complex Monge-Amp\`ere eigenvalue problem.
\begin{lem}
\label{allmono}
Let $\Omega$ be a bounded convex domain in $\R^n$ with nonempty interior. Let $u_0 \in C(\overline{\Omega})$ be a nonzero convex function on $\Omega$ with $u_0=0$ on $\p\Omega$  and $0<R(u_0)<\infty$.
Consider the iterative scheme (\ref{IIS}). Then for all $k\geq 0$, we have
\begin{enumerate}
\item $E(u_k)\leq E(u_{k+1})$;
\item $\|u_k\|_{L^{n+1}(\Omega)} \leq \|u_{k+1}\|_{L^{n+1}(\Omega)}$; and
\item $R(u_{k+1}) \leq R(u_k)$.
\end{enumerate}

\end{lem} 
\begin{proof} Multiplying both sides of the first equation of (\ref{IIS}) by $|u_k|$, integrating over $\Omega$ and then using Lemma \ref{NCS}, we have
\[E(u_k)= R(u_k) \|u_k\|^{n+1}_{L^{n+1}(\Omega)}=\int_{\Omega} |u_k|\det D^2 u_{k+1}\,dx \leq 
[E(u_k)]^{\frac{1}{n+1}} [E(u_{k+1})]^{\frac{n}{n+1}}.
\]
Then, (i) easily follows.  

On the other hand, 
since \[R(u_k) \|u_k\|^n_{L^{n+1}(\Omega)} =  \frac{E(u_k)}{ \|u_k\|_{L^{n+1}(\Omega)}},\]
we can rewrite the monotonicity property of Lemma \ref{AKR} as
\[\frac{E(u_{k+1})}{ \|u_{k+1}\|_{L^{n+1}(\Omega)}} \leq \frac{E(u_k)}{ \|u_k\|_{L^{n+1}(\Omega)}}. \]
This combined with (i) gives (ii).

Finally, we deduce (iii) from (ii) and Lemma \ref{AKR}.
\end{proof}
We are now in a position to prove Theorem \ref{lamC}.
\begin{proof}[Proof of Theorem \ref{lamC}]

Assume $u\in C(\overline{\Omega})$ is a nonzero convex function on $\Omega$ with $u=0$ on $\p\Omega$ that satisfies
\begin{equation} \label{Eeq} \int_\Omega |u|\det D^2 u\, dx =\lambda[\Omega] \int_\Omega |u|^{n+1}\, dx.\end{equation}

Consider the iterative scheme (\ref{IIS}) with initial data $u_0=u$. Then 
\[
R(u_0)=\lambda[\Omega].
\]
Recalling \eqref{lam_def}, we have 
\[R(u_k)\geq \lambda[\Omega]\quad \text{for all } k\geq 0.\]
These combined with the nonincreasing property of $\big\{R(u_k)\big\}_{k=0}^{\infty}$ established in Lemma \ref{allmono} imply that
\[R(u_k)= \lambda[\Omega]\quad \text{for all } k\geq 0.\]
Therefore, we must have an equality in (ii) of Lemma \ref{allmono}
\begin{equation}
\label{kkeq}
\|u_k\|_{L^{n+1}(\Omega)} = \|u_{k+1}\|_{L^{n+1}(\Omega)}\quad\quad \text{for all } k\geq 0,
\end{equation}
and the inequality in Lemma \ref{AKR} must be an equality for all $k\geq 0$.  From its proof using H\"older inequality, we see that this happens if and only if there is a constant $c_k> 0$ such that
\[|u_{k+1}|= c_k |u_k|\quad \text{for all } k\geq 0.\]  From \eqref{kkeq}, we deduce that $c_k=1$ for all $k\geq 0$. It follows that 
\[u_k= u_0=u \quad \text{for all } k\geq 0.\]
By Theorem \ref{IISthm}, 
$\{u_k\}$ converges uniformly on $\overline{\Omega}$ to a nonzero Monge--Amp\`ere eigenfunction of $\Omega$, so 
$u$ must be a nonzero Monge--Amp\`ere eigenfunction of $\Omega$.
The theorem is proved.
\end{proof}

 {\bf Acknowledgements.} I would like to thank Farhan Abedin, Jun Kitagawa and the referees for their critical comments and helpful suggestions 
 that help improve the exposition, and strengthen the results 
 of the note.

 \section*{Declarations}
{\bf Conflict of Interest Statement: } There is no conflict of interest.\\

 {\bf Data Availability Statement:}  Data sharing does not apply to this article as no datasets were generated or analyzed during the current study.

\end{document}